\documentclass[12pt, reqno]{amsart}
\usepackage{amsmath, amsthm, amscd, amsfonts, amssymb, graphicx, color}
\usepackage{mathrsfs}
\usepackage{graphicx}
\usepackage{caption}
\usepackage{subcaption}
\usepackage[bookmarksnumbered, colorlinks, plainpages]{hyperref}
\hypersetup{colorlinks=true,linkcolor=red, anchorcolor=green, citecolor=cyan, urlcolor=red, filecolor=magenta, pdftoolbar=true}

\newtheorem{theorem}{Theorem}[section]
\newtheorem{lemma}[theorem]{lemma}

\theoremstyle{remark}

\numberwithin{equation}{section}
\allowdisplaybreaks
\begin{document}

\title[]{Logarithmic Coefficients Problems of  Geometric Subclass of Closed-to-convex Functions}

\author[C. Dhara, and  N. Ghosh]{Chayani Dhara, and Nirupam Ghosh}
	
\address[Dhara]{Department of Mathematics, Indian Institute of Engineering Science and Technology, Shibpur, Howrah 711103, West Bengal, India}
\email{chayanidhara1999@gmail.com}

\address[Ghosh] {Department of Mathematics, Indian Institute of Engineering Science and Technology, Shibpur, Howrah 711103, West Bengal, India}
\email{nirupamghoshmath@gmail.com}

\subjclass[2020]{30C45, 30C50, 30C80}

\keywords{Univalent, Close--to-convex, Differential subordination, Coefficient estimates, Logarithmic coefficients, Hankel determinant}

\begin{abstract} 
	 For $\alpha\ge 0$, let $\mathcal{W}(\alpha)$ be the class of all analytic functions in the unit disk $\mathbb{D}$ with normalization  $f(0) = 0 $ and $ f'(0) = 1 $ that satisfy the relation $Re\,\{f'(z) + \alpha z f''(z)\} > 0$. This article aims to establish sharp bounds for logarithmic coefficients $\gamma_1$, $\gamma_2$ and $\gamma_3$ and  logarithmic inverse coefficients $\Gamma_1$, $\Gamma_2$ and $\Gamma_3$  of functions in $\mathcal{W}(\alpha)$. The sharp upper and lower bounds for $\bigl|\,\gamma_2 \,\bigr|-\bigl|\,\gamma_1\,\bigr|$ and $\bigl|\,\Gamma_2 \,\bigr|-\bigl|\,\Gamma_1\,\bigr|$  have been obtained for the class  $\mathcal{W}{(\alpha)}$. In addition, we establish sharp inequality for the second Hankel determinant of the logarithmic and  inverse logarithmic coefficients for the class $\mathcal{W}{(1)}$.
\end{abstract}

\maketitle	

\section{Introduction}
Let  $\mathcal{H}$ be the class of holomorphic functions defined on the unit disk $\mathbb{D}:=\{z\in\mathbb{C}:|z|<1\}$. Let $\mathcal{A}$ be the class of functions $f\in\mathcal{H}$ such that $f(0) = 0 $ and $ f'(0) = 1 $. Any functions $f \in \mathcal{A}$ have the following power series expansion
\begin{equation}\label{P1-eq-01}
f(z)=z+\sum_{n=2}^{\infty} a_n z^n~~~~~~~~~~~~~~~~ z \in \mathbb{D}.
\end{equation}
 The most familiar subclass of $\mathcal{A}$ consists of univalent (i.e, one-to-one) functions $f$ in $\mathbb{D}$, and it is denoted by $\mathcal{S}$. 
Let $\mathcal{S}^*$ and $\mathcal{K}$ be the classes of starlike and close-to-convex functions in $\mathcal{S}$, respectively. A function $f\in\mathcal{A}$ is said to be a starlike function in $\mathbb{D}$ if $f(\mathbb{D})$ is a starlike domain with respect to the origin. A function $f\in \mathcal{S}^{*}$  if and only if
$$Re\,\left(z\frac{f'(z)}{f(z)}\right)>0~~ \text{for}~ z\in \mathbb{D}.$$
A function $f\in\mathcal{A}$ is close-to-convex if $\mathbb{C} \setminus f(\mathbb{D})$ can be written as the union of non-intersecting half-lines. It is known (see \cite[Theorem 2.16]{duren1983univalent}) that a function $f \in\mathcal{K}$  if and only if  there exist $g\in\mathcal{\mathcal{S}^*}$ and $\beta\in (-\pi/2, \pi/2)$ such that
$$Re\,\left(e^{i\beta}\frac{z f'(z)}{g(z)}\right) >0 ~~ \text{for}~ z\in \mathbb{D}.$$

In 1977, Chichra \cite{chichra1977newclass} introduced the class $\mathcal{W}{(\alpha)}$, consisting of functions $f\in\mathcal{A}$ such that $Re\,\{f'(z) + \alpha z f''(z)\} > 0,~~~ \mbox{for}~~~z \in \mathbb{D}~~~ \mbox{and}~~~\alpha \ge 0.$
Chichra \cite{chichra1977newclass} shown that 
 $\mathcal{W}{(\alpha)} \subset \mathcal{K}$, and consequently 
$\mathcal{W}{(\alpha)}$ forms a subclass of $\mathcal{S}$. In 1982, Singh and Singh \cite{Singh1982} proved that the class 
$\mathcal{W}(1)$ is contained in the class of starlike functions 
$\mathcal{S}^*$. Subsequently, Ponnusamy and Vasudevarao \cite{Ponnusamy2010} investigated the variability regions associated with functions belonging to 
$\mathcal{W}{(\alpha)}$. 

 For each $f \in \mathcal{S}$ there exists a unique function $F_{f}\in\mathcal{A} $ such that $ F_{f}(0)=0$ and $$f(z)=z\exp(F_{f}(z)), ~~~z \in \mathbb{D}.$$ 
 Moreover, for some neighborhood of the origin, $F_{f}$ can be expressed as 
 \begin{equation}\label{P1-eq-01A}
   F_{f}(z)=\log \frac{f(z)}{z} = 2\sum_{n=1}^{\infty} \gamma_n z^n,   
 \end{equation}
where $\gamma_n,~n \in \mathbb{N}$ are known as logarithmic coefficients of $f \in \mathcal{S}$.
 Logarithmic coefficients are known to play a fundamental role in the Milin conjecture(\cite{Milin1977} for $f \in \mathcal{S}$.  Louis de Branges \cite{Branges1985APO} solved the long-standing Bieberbach conjecture \cite{bieberbach1916uber} by showing the Milin conjecture\cite{Milin1977}  for $f \in \mathcal{S}$. Milin conjected that for $f \in \mathcal{S}$ and $n\geq 2$
 $$
 \sum_{m=1}^{n} \sum_{k=1}^{m} \left( k |\gamma_k|^2 - \frac{1}{k} \right) \le 0,
 $$
 where equality holds if and only if $f$ is the Koebe function $K(z) = z / (1 - z)^2$ or its rotation $e^{-i \theta }K( e^{i \theta }z) = z / (1 - e^{i \theta }z)$ for some $\theta \in \mathbb{R}$. For the Koebe function, the logarithmic coefficients are $\gamma_n = {1}/{n}$.
Since the Koebe function frequently arises as the extremal function for a wide range of extremal problems for the class $\mathcal{S}$, it is natural to expect that $|\gamma_n| \le {1}/{n}$ holds for functions in $\mathcal{S}$. However, this bound does not hold in general, even at the order-of-magnitude level. Indeed, there exists a bounded function $f \in \mathcal{S}$ with logarithmic coefficients $\gamma_n = O(n^{-0.83})$ (see \cite[Theorem 8.4]{duren1983univalent}).

Note that differentiating (\ref{P1-eq-01A}) and using (\ref{P1-eq-01}), we get
\begin{align}\label{p1-eq-002}
& \gamma_1 = \frac{1}{2} a_2, \\ \nonumber
& \gamma_2 = \frac{1}{2}\left( a_3 - \frac{1}{2} a_2^2 \right),\\\nonumber
&\gamma_3 = \frac{1}{2}\left( a_4 - a_2 a_3 + \frac{1}{3} a_2^3 \right).
\end{align}
As $|a_2| \leq 2$, for $f \in \mathcal{S}$, $|\gamma_1| \leq 1$. Using the Fekete-Szeg\"o inequality \cite [Theorem 3.8]{duren1983univalent} for $ f \in \mathcal{S}$, we can get the sharp estimate 
$$
|\gamma_{2}| \leq \frac{1}{2} (1 + 2 e^{-1})| = 0.635\ldots.
$$
For $n \geq 3$, no significant upper bound for $|\gamma_n|$ where $f \in \mathcal{S}$ is yet to be known. Logarithmic coefficients for the class $\mathcal{S}$ and its subclasses are one of the recent topics of interest for various authors. Significant results in this topic can be found in  \cite{Firoz2018, Duren1979, Elhosh1996, Girela2000, LeckoSim2025} and the references there in.

Let $F$ be the inverse function of $f \in \mathcal{S}$ defined in a neighborhood of the origin with the Taylor
series expansion  
\begin{equation}\label{eq:inverse-expansion}
    F(w) := f^{-1}(w) = w + \sum_{n=2}^{\infty} A_n w^n, ~~~~|w| < 1/4.
\end{equation}
Using a variational method, Löwner \cite{Loewner1923} obtained the sharp estimate $|A_n| \le K_n$ for each $n \in \mathbb{N}$, where
$K_n = {(2n)!}/({n!(n+1)!})$ and $K(w) = w + K_2 w^2 + K_3 w^3 + \cdots$ is the inverse of the
Köebe function. There has been a good deal of interest in determining the behavior of the inverse coefficients
of $f$ given by (\ref{P1-eq-01}) where the corresponding function $f$ restricted to some proper geometric subclasses of $\mathcal{S}$. Since $f(f^{-1}(w)) = w$, using \eqref{eq:inverse-expansion}, we get
\begin{align}\label{eq:inverse-coeff-relations}
   A_2 = -a_2,~~~~~A_3 = -a_3 + 2a_2^2 ~~~~~\mbox{and}~~~~~A_4 &= -a_4 + 5a_2 a_3 - 5 a_2^3.
\end{align}

For $f\in \mathcal{S}$, the logarithmic inverse coefficients are defined by the equation
\begin{equation}\label{eq:log-inverse-def}
    \log \frac{f^{-1}(w)}{w}
    = 2\sum_{n=1}^{\infty} \Gamma_n\, w^n,
    \qquad |w|<\frac{1}{4}.
\end{equation}
Differentiating \eqref{eq:log-inverse-def} together with \eqref{eq:inverse-expansion} and \eqref{eq:inverse-coeff-relations} we get the relation between the logarithmic inverse coefficients and Taylor's coefficients of $f \in \mathcal{S}$ given by (\ref{P1-eq-01}) as 
\begin{equation}\label{eq:Gamma-coefficients}
    \Gamma_1 = -\frac{1}{2}a_2,~~
    \Gamma_2 = -\frac{1}{2}\left(a_3 - \frac{3}{2}a_2^2\right) ~~\mbox{and}~~
    \Gamma_3 = -\frac{1}{2}\left(a_4 - 4a_2 a_3 + \frac{10}{3}a_2^3\right).
\end{equation}
In \cite{Ponnusamy2018} Ponnusamy et al. obtained a sharp bound for the logarithmic inverse coefficients for the class $\mathcal{S}$. In fact, Ponnusamy et al. \cite{Ponnusamy2018} established for $f \in \mathcal{S}$ that
$$|\Gamma_n| \le \frac{1}{2n}
    \binom{2n}{n},$$
and showed that the equality holds only for the Köebe function or its rotations.  Moreover Ponnusamy et al. \cite{Ponnusamy2018} obtained sharp bound for logarithmic inverse coefficients for some of the important geometric subclasses of $\mathcal{S}$.



For $q, n \in \mathbb{N}$, the Hankel determinant $H_{q,n}(f)$ of the Taylor's coefficients of a
function $f \in \mathcal{A}$ of the form \eqref{P1-eq-01} is defined by
$$
H_{q,n}(f) =
\begin{vmatrix}
a_n & a_{n+1} & \cdots & a_{n+q-1} \\
a_{n+1} & a_{n+2} & \cdots & a_{n+q} \\
\vdots & \vdots & \ddots & \vdots \\
a_{n+q-1} & a_{n+q} & \cdots & a_{n+2(q-1)}
\end{vmatrix}.
$$
Hankel determinants of various orders have recently been studied with considerable attention and have been investigated by several authors (see \cite{ref12, Krishna2015, ref5, ref18,  Sim2022}). It is well known that the Fekete-Szegö functional corresponds to the second Hankel determinant $H_{2,1}(f)$. Fekete-Szegö then further generalized  by establishing bounds for  $\lvert a_3 - \mu a_2^2 \rvert$
with $\mu \in \mathbb{R}$  for $f \in \mathcal{S}$ (see \cite[Theorem 3.8]{duren1983univalent}). Recently increasing attention has been devoted to the investigation of Hankel determinants involving logarithmic coefficients (denoted by $H_{q,n}(F_f/2) $)  and logarithmic inverse coefficients (denoted by $H_{q,n}(F_{f^{-1}}/2)$) for various geometric subclasses of $\mathcal{S}$ (see \cite{ref3, ref6, ref13,ref14}).


In 2023,   Lecko and Partyka \cite{LeckoPartyka2023} investigated sharp upper and lower bounds for $|\gamma_2| - |\gamma_1|$ for functions in the class $\mathcal{S}$. Also Kumar and Cho~\cite{KumarCho2023} established sharp bounds for $\bigl|\,\gamma_2 \,\bigr|-\bigl|\,\gamma_1\,\bigr|$ within several subclasses of $\mathcal{S}$. More recently, Allu and Shaji \cite{Allu2025} estimated the sharp upper and lower bounds for moduli difference $\bigl|\,\Gamma_2 \,\bigr|-\bigl|\,\Gamma_1\,\bigr|$ for functions in class $\mathcal{S}$ and for functions in some important subclasses of univalent functions.

 This paper is organized as follows. In section \ref{section-2}, we have discussed some basic results to prove our main results. In section \ref{section-3}, we have proved the logarithmic coefficients and logarithmic inverse coefficients for functions in $\mathcal{W}(\alpha)$. In section \ref{section-4}, we have discussed the moduli difference of successive logarithmic and logarithmic inverse coefficients for functions in $\mathcal{W}(\alpha)$. In section \ref{section-5}, we have investigated the Hankel determinant whose elements are logarithmic and logarithmic inverse coefficients of $f \in \mathcal{W}(1)$.

\section{Preliminary Results}\label{section-2}

Let $\mathcal{P}$ be the class of all analytic functions $p$ in the unit disk 
$\mathbb{D} = \{ z \in \mathbb{C} : |z| < 1 \}$ satisfying $p(0) = 1$ and $Re\,( p(z)) > 0$ for $z \in \mathbb{D}$.
Therefore, every $p \in \mathcal{P}$ can be represented as 
\begin{equation}\label{P Class-1}
 p(z) = 1 +\sum_{n=1}^{\infty} p_n z^n ~~~~~~~~~~~~z \in \mathbb{D}.   
\end{equation}
Elements of the class $\mathcal{P}$ are called \textit{Carath\'eodory functions}. 
It is well known that $|p_n| \leq 2$, $n \geq 1$, for a function $p \in \mathcal{P}$ (see \cite{duren1983univalent}).
The Carathéodory class $\mathcal{P}$ appears naturally in the study of univalent functions. 

To prove our results, we need the following lemmas for  the Carathéodory class $\mathcal{P}$.
\begin{lemma}\cite{MaMinda1992} \label{p1-lemma-001}
Let $p \in \mathcal{P}$ be of the form (\ref{P Class-1}). Then for a real number $v$,
$$
\left| p_2 - v p_1^2 \right|
\le
\begin{cases}
-4v+2, & \text{if } v \le 0, \\[6pt]
2, & \text{if } 0 \le v \le 1, \\[6pt]
4v-2, & \text{if } v \ge 1 .
\end{cases}$$
\end{lemma}

Moreover, equality holds under the following conditions:

\begin{itemize}
\item[(i)] If $v<0$ or $v>1$, equality holds if and only if
$$
p(z)=\frac{1+z}{1-z}
\quad \text{or one of its rotations.}
$$

\item[(ii)] If $0<v<1$, equality holds if and only if
$$p(z)=\frac{1+z^2}{1-z^2}\quad \text{or one of its rotations.}$$

\item[(iii)] If $v=0$, equality holds if and only if
$$
p(z)=\left(\frac12+\frac{\lambda}{2}\right)\frac{1+z}{1-z}
+\left(\frac12-\frac{\lambda}{2}\right)\frac{1-z}{1+z},
\qquad 0\le \lambda \le 1,$$or one of its rotations.

\item[(iv)] If $v=1$, equality holds only when $p$ is the reciprocal of one of the functions that guarantees equality in the case $v=0$, and only in that case.
\end{itemize}

\begin{lemma}\cite{SimThomas2020}\label{p1-lemma-005}
Let $p \in \mathcal{P}$ be of the form (\ref{P Class-1}) and  $B_1, B_2,$ and $B_3$ be numbers such that $B_1 \ge 0$,  $B_2 \in \mathbb{C}$, and $B_3 \in \mathbb{R}$. 
Define $\Psi_{+}(p_1,p_2)$ and $\Psi_{-}(p_1,p_2)$ by
$$
\Psi_{+}(p_1,p_2) = \left| B_2 p_1^2 + B_3 p_2 \right| - \left| B_1 p_1 \right|,
$$
and
$$
\Psi_{-}(p_1,p_2) = -\Psi_{+}(p_1,p_2).
$$
Then
$$
\Psi_{+}(p_1,p_2) \le
\begin{cases}
|4B_2 + 2B_3| - 2B_1, & \text{when } |2B_2 + B_3| \ge |B_3| + B_1, \\[6pt]
2|B_3|, & \text{otherwise},
\end{cases}
$$
and
$$
\Psi_{-}(p_1,p_2) \le
\begin{cases}
2B_1 - B_4, & \text{when } B_1 \ge B_4 + 2|B_3|, \\[8pt]
\frac{2B_1 \sqrt{2|B_3|}}{\sqrt{B_4 + 2|B_3|}}, 
& \text{when } B_1^2 \le 2|B_3|(B_4 + 2|B_3|), \\[12pt]
2|B_3| + \frac{B_1^2}{B_4 + 2|B_3|}, 
& \text{otherwise},
\end{cases}
$$
where
$$
B_4 = |4B_2 + 2B_3|.
$$
All inequalities are sharp.
\end{lemma}

Let $\mathcal{B}$ denote the class of analytic functions $\varphi : \mathbb{D} \to \mathbb{D}$ such that $\varphi(0)=0$. Functions in $\mathcal{B}$ are known as the Schwarz functions.  
A function $\varphi \in \mathcal{B}$ can be written as a power series
\begin{equation}\label{Schwarz function-1}
    \varphi(z)=\sum_{n=1}^{\infty} c_{n} z^{n}.
\end{equation}

For $f, g \in \mathcal{H}$, we say that $f$ is \emph{subordinate} to $g$, denoted by
$f \prec g$, if there exists a  Schwarz function $\varphi$ such that $$f(z)=g(\varphi(z)) ~~~~~ z \in \mathbb{D}.$$ 
In particular, if $g$ is
univalent, then $f \prec g$ if, and only if, $f(0)=g(0)$ and $f(\mathbb{D}) \subset g(\mathbb{D})$.

Next we recall the following well-known result for Schwarz functions.

\begin{lemma}\cite{Carlson1940}\label{p1-lemma-010}
Let $\varphi(z)=c_{1}z + c_{2}z^{2} + c_{3}z^{3} + \cdots$ be a Schwarz function.  
Then
\begin{equation*}
|c_{1}| \le 1,~~~ |c_{2}| \le 1 - |c_{1}|^{2},~~~|c_3|\le 1 - |c_{1}|^{2}-\frac{|c_{2}|^{2}}{ 1 + |c_{1}|}.    
\end{equation*}
\end{lemma}

In 1981, Prokhorov et al. \cite{ProkhorovSzynal1981} obtained a coefficient inequality for functions in $\mathcal{B}$.

\begin{lemma}\cite{ProkhorovSzynal1981}\label{p1-lemma-015}
Let $\varphi(z)\in \mathcal{B}$ be of the form (\ref{Schwarz function-1}). Then for any real $\mu $ and $\nu$,we have 
$$
|c_3+\mu\,c_1c_2+\nu\,c_1^{3}|=
\begin{cases}
1, & \text{if } (\mu,\nu)\in D_1\cup D_2 \cup \{(2,1)\}, \\[6pt]
|\nu|, & \text{if } (\mu,\nu)\in \displaystyle\bigcup_{k=3}^{7} D_k, \\[8pt]
\end{cases}
$$
The regions $D_k  ~~ \mbox{for}~~~k=1,\dots,7$ are defined as follows:
\begin{align*}
& D_1=\{(\mu,\nu):|\mu|\le \tfrac12,\;|\nu|\le 1\},\\
& D_2=\{(\mu,\nu):\tfrac12\le|\mu|\le 2,\; \tfrac{4}{27}(|\mu|+1)^3-(|\mu|+1)\le|\nu|\le 1\},\\
& D_3=\{(\mu,\nu):|\mu|\le\tfrac12,\;\nu\le -1\},\\
& D_4=\{(\mu,\nu):|\mu|\ge\tfrac12,\;\nu\le -\tfrac23(|\mu|+1)\},\\
& D_5=\{(\mu,\nu):|\mu|\le 2,\;|\nu|\ge 1\},\\
& D_6=\{(\mu,\nu):2\le|\mu|\le 4,\;|\nu|\ge \tfrac1{12}\sqrt{\mu^2+8}\},\\
& D_7=\{(\mu,\nu):|\mu|\ge 4,\;|\nu|\ge \tfrac23(|\mu|-1)\}. 
\end{align*}
Moreover, all inequalities are sharp.
\end{lemma}

\section{Logarithmic Coefficients and Logarithmic Inverse Coefficients}\label{section-3}

Let   $ f \in \mathcal{W}{(\alpha)}$ for $\alpha \geq 0$, be of  the form \eqref{P1-eq-01}. Therefore, there exist a function  $p \in \mathcal{P}$ of the form \eqref{P Class-1} such that
\begin{equation}\label{P1-eq-mr-1}
 f'(z)+\alpha z f''(z) = p(z).   
\end{equation}
Also in view of subordination, there is a Schwarz function $\varphi \in \mathcal{B}$ of the form \eqref{Schwarz function-1}, such that 
\begin{equation}\label{P1-eq-mr-02}
  p(z) = \frac{1 + \varphi(z)}{ 1 - \varphi(z)}.  
\end{equation}
Using the power series representation of $f$ and $p$ and comparing the coefficients of $z^n$, from \eqref{P1-eq-mr-1} we get 
\begin{equation}\label{P1-eq-mr-10}
    (n + 1)( 1 + n \alpha)a_{n + 1}  = p_n, ~~~~~~~n\ge 1. 
\end{equation}
So by Carathéodory estimate $|p_n|\le 2$, for $n \geq 1$, from \eqref{P1-eq-mr-10}, we obtain 
\begin{equation}
|a_{n+1}|\le \frac{2}{(n+1)(1+n\alpha)}, \qquad n\ge 1 
\end{equation}
for any function $f \in \mathcal{W}(\alpha)$.

With the help of the observation above, we find the sharp bounds of the logarithmic coefficients for functions in the class $\mathcal{W}{(\alpha)}$ for $ \alpha \geq 0 $. Here we want to mention that, in 2022 Anand et. al \cite{Anand-2022} discussed the bounds of first three logarithmic coefficients (i.e, $|\gamma_1|, |\gamma_2|$ and $|\gamma_3|$) for the functions in the class $\mathcal{W}(\alpha)$ for $Re\,({\alpha}) \geq 0$. However, they don't discuss about the extreme function for $|\gamma_3|$ explicitly, and   the exact forms of the extreme functions for $|\gamma_1|$ and $|\gamma_2|$ are not mentioned.  
In the first main result of this section, we determine the sharp bounds $|\gamma_1|, |\gamma_2|$, $|\gamma_3|$ and construct extreme functions explicitly.

\begin{theorem}
Let $f(z)\in \mathcal{W}(\alpha)$ for $\alpha \geq 0$, be of the form \eqref{P1-eq-01}. Then  the first three logarithmic coefficients satisfy the sharp estimates 
\begin{align*}
    & |\gamma_1|\le \dfrac{1}{2(1+\alpha)},\\
    & |\gamma_2|\le \dfrac{1}{3(1+2\alpha)},\\
    & |\gamma_3|\le \dfrac{1}{4(1+3\alpha)}.
\end{align*}
\end{theorem}

\begin{proof}

Let $f\in\mathcal{W}{(\alpha)}$ be of the form \eqref{P1-eq-01}. Thus  from \eqref{P1-eq-mr-10} we get
\begin{equation} \label{P1-eq-mr-15}
 a_{n+1}=\frac{p_n}{(n+1)(1+n\alpha)},\qquad n\ge1.   
\end{equation}

In view of \eqref{p1-eq-002} and \eqref{P1-eq-mr-15} we get 
\begin{align} 
& \gamma_1=\frac{p_1}{4(1+\alpha)}, ~~~~~~ \gamma_2=\frac{p_2}{6(1+2\alpha)}-\frac{p_1^2}{16(1+\alpha)^2}, \label{P1-eq-mr-20}\\
& \gamma_3=\frac{p_3}{8(1+3\alpha)}-\frac{p_1p_2}{12(1+\alpha)(1+2\alpha)}
+\frac{p_1^3}{48(1+\alpha)^3}. \label{P1-eq-mr-25}
\end{align}
Since $|p_n|\le 2$ for $n \geq 1$, from \eqref{P1-eq-mr-20} we get 
$$
|\gamma_1|\le \frac{2}{4(1+\alpha)}=\frac{1}{2(1+\alpha)},
$$
and equality holds for $ f_1 \in \mathcal{W}(\alpha)$ defined by 
\begin{equation}\label{P1-eq-mr-26}
f_1(z)=z+\sum_{n=2}^{\infty}\dfrac{2}{n\big(1+(n-1)\alpha\big)}\,z^n.    
\end{equation}
Furthermore,  from \eqref{P1-eq-mr-20} we get 
$$|\gamma_2|=\frac{1}{6(1+2\alpha)}\left|p_2-\frac{3(1+2\alpha)}{8(1+\alpha)^2}\,p_1^2\right|.$$
For $\alpha \geq 0$, it is easy to check that 
$$
\frac{3(1+2\alpha)}{8(1+\alpha)^2} = \frac{3}{8(1 + \alpha)} + \frac{3 \alpha}{8(1 + \alpha)^2} \leq \frac{3}{4} < 1. 
$$
Therefore, in view of  Lemma \ref{p1-lemma-001}, we get
$$|\gamma_2|\le\frac{2}{6(1+2\alpha)} = \frac{1}{3(1+2\alpha)}, $$
and  equality holds for $f_2 \in \mathcal{W}(\alpha)$ defined by
\begin{equation}\label{P1-eq-mr-27}
 f_2(z)=z+\sum_{n=1}^{\infty}\dfrac{2}{(2n+1)(1+2n\alpha) }z^{2n+1}.   
\end{equation}
Also from \eqref{P1-eq-mr-25}, we get 
\begin{equation}\label{P1-eq-mr-30}
 |\gamma_3|= \left|\frac{p_3}{8(1+3\alpha)}-\frac{p_1p_2}{12(1+\alpha)(1+2\alpha)}
+\frac{p_1^3}{48(1+\alpha)^3}\right|.   
\end{equation}
It is observed that by equating the coefficients of $z, z^2$ and $z^3$ in \eqref{P1-eq-mr-02}, we obtain 
\begin{equation}\label{P1-eq-mr-34}
 p_1=2c_1,~~~~p_2=2(c_2+c_1^2), ~~~~p_3=2(c_3+2c_1c_2+c_1^3).
 \end{equation}
Using \eqref{P1-eq-mr-34} in \eqref{P1-eq-mr-30}, we get
\begin{align}\label{P1-eq-mr-35}
|\gamma_3
|& =\left|\frac{c_3+2c_1c_2+c_1^3}{4(1+3\alpha)}
-\frac{c_1c_2+c_1^3}{3(1+\alpha)(1+2\alpha)}
+\frac{c_1^3}{6(1+\alpha)^3}\right|\\ \nonumber 
& = \dfrac{1}{4(1+3\alpha)}
\left|c_3+ \mu c_1 c_2 + \nu c_1^{3}\right|,
\end{align}
where
\begin{align*}
 & \mu= 2 - \dfrac{4(1+3\alpha)}{3(1+\alpha)(1+2\alpha)}\\ 
 & \nu = 1 - \dfrac{4(1+3\alpha)}{3(1+\alpha)(1+2\alpha)}
  + \dfrac{2(1+3\alpha)}{3(1+\alpha)^{3}}.
\end{align*}
We now use Lemma  \ref{p1-lemma-015} to  estimate the third logarithmic coefficient $\gamma_{3}$.
For $\alpha \geq 0$, a lengthy but straightforward calculation shows that 
$$
\frac{2}{3} \leq \mu < 2 ~~~~ \mbox{and}~~~~\dfrac{4}{27}(|\mu|+1)^3-(|\mu|+1)\le|\nu|\le 1.
$$
Thus $(\mu, \nu) \in D_2$ defined as in   Lemma  \ref{p1-lemma-015}. Hence 
$$
\left|c_3+ \mu c_1 c_2 + \nu c_1^{3}\right| \leq 1.
$$
Therefore, from \eqref{P1-eq-mr-35},  we get $$|\gamma_3|
\le\dfrac{1}{4(1+3\alpha)},$$
and the equality holds for the function  $f_3 \in \mathcal{W}(\alpha)$ defined by 
\begin{equation}\label{P1-eq-mr-36}
 f_3(z)=z+\sum_{n=1}^{\infty}\frac{2}{(3n+1)(1+3n\alpha)}\,z^{3n+1}.   
\end{equation}
\end{proof}

Now we shall discuss the logarithmic inverse coefficients of the functions in $\mathcal{W}(\alpha)$ for $\alpha \geq 0$.

\begin{theorem}
Let  $f \in \mathcal{W}(\alpha)$ for $\alpha \geq 0$, be of the form \eqref{P1-eq-01} and $F$ be the inverse function of $f$ have the form  \eqref{eq:inverse-expansion}. Then the first three logarithmic inverse coefficients of $f$ satisfy the sharp estimates 
\begin{align*}
& |\Gamma_1|\le \frac{1}{2(1+\alpha)}, ~~~~~~~~~~~~ \alpha\geq 0 ,\\
& |\Gamma_2|\le 
\begin{cases}
\dfrac{5+10\alpha-4\alpha^2}{12(1+2\alpha)(1+\alpha)^2}, & 0\le\alpha < 1/2, \\
  \dfrac{1}{3(1+2\alpha)}, & \alpha \ge\,1/2,  
\end{cases}\\
&|\Gamma_3|
\le \begin{cases}
    \dfrac{1}{4(1+3\alpha)}-\dfrac{4}{3(1+\alpha)(1+2\alpha)}+\dfrac{5}{3(1+\alpha)^3}, & 0 \leq  \alpha\le 0.8090,\\
  \dfrac{1}{4(1+3\alpha)}, & \alpha> 0.8090.
\end{cases}
\end{align*}
\end{theorem}

\begin{proof}
Let  $f\in\mathcal{W}{(\alpha)}$  be of the form \eqref{P1-eq-01}. Therefore from \eqref{P1-eq-mr-10} we get 
\begin{equation}\label{P1-eq-mr-40}
 a_{n+1}=\frac{p_n}{(n+1)(1+n\alpha)} ~~~~~n\ge1.   
\end{equation}
In view of \eqref{eq:Gamma-coefficients} and \eqref{P1-eq-mr-40}
we obtain
\begin{align}
 & \Gamma_1=-\frac{p_1}{4(1+\alpha)} \label{P1-eq-mr-45}\\  
  &\Gamma_2=-\frac{p_2}{6(1+2\alpha)}+\frac{3p_1^2}{16(1+\alpha)^2}  \label{P1-eq-mr-50}\\
 & \Gamma_3=-\frac{p_3}{8(1+3\alpha)}+\frac{p_1p_2}{3(1+\alpha)(1+2\alpha)} -\frac{5p_1^3}{24(1+\alpha)^3}.\label{P1-eq-mr-55}
\end{align}
Since $|p_1|\le 2$, from \eqref{P1-eq-mr-45}, we get 
$$\
|\Gamma_1|\le \frac{2}{4(1+\alpha)}=\frac{1}{2(1+\alpha)},
$$
and equality holds for the inverse function of  $f_1 \in \mathcal{W}(\alpha)$ defined by \eqref{P1-eq-mr-26}.

From \eqref{P1-eq-mr-50}, we get 
\begin{align}\label{P1-eq-mr-60}
|\Gamma_2| & = \left|-\frac{p_2}{6(1+2\alpha)}+\frac{3p_1^2}{16(1+\alpha)^2} \right|\\ \nonumber 
& = \frac{1}{6(1+2\alpha)}\left|p_{2}-\frac{9(1+2\alpha)}{8(1+\alpha)^2}\,p_{1}^2\right|.
\end{align}
 For $\alpha \geq 0$, one can verify that 
\begin{equation}\label{P1-eq-mr-65}
  \frac{9(1+2\alpha)}{8(1+\alpha)^2} \leq 1 ~~~ \mbox{for}~~~ 0\leq \alpha \leq \frac{1}{2} ~~~\mbox{and} ~~~\frac{9(1+2\alpha)}{8(1+\alpha)^2} > 1 ~~~ \mbox{for} ~~~ \alpha > \frac{1}{2}.  
\end{equation}
By Lemma \ref{p1-lemma-001} and \eqref{P1-eq-mr-65}, from \eqref{P1-eq-mr-60} we get 
\begin{align*}
|\Gamma_2| &  =   \frac{1}{6(1+2\alpha)}\left|p_{2}-\frac{9(1+2\alpha)}{8(1+\alpha)^2}\,p_{1}^2\right|\\
& \leq   \frac{1}{6(1+2\alpha)} \left( \frac{9(1+2\alpha)}{8(1+\alpha)^2} - 2\right) \\ 
& = \frac{5+10\alpha-4\alpha^2}{12(1+2\alpha)(1+\alpha)^2} ~~~ \mbox{for}~~~~ 0 \le\alpha \leq  1/2,
\end{align*}
and equality holds for the inverse function of  $f_1 \in \mathcal{W}(\alpha)$ defined by \eqref{P1-eq-mr-26}.
Also by Lemma \ref{p1-lemma-001} and \eqref{P1-eq-mr-65}, from \eqref{P1-eq-mr-60} we get
\begin{align*}
|\Gamma_2| &  =   \frac{1}{6(1+2\alpha)}\left|p_{2}-\frac{9(1+2\alpha)}{8(1+\alpha)^2}\,p_{1}^2\right|\\
& \leq  \frac{1}{3(1+2\alpha)} ~~~ \mbox{for}~~~\alpha>  1/2,
\end{align*}
and equality holds for the inverse function of $f_2 \in \mathcal{W}(\alpha)$ defined by  \eqref{P1-eq-mr-27}.


From \eqref{P1-eq-mr-55}, We get,
\begin{equation}\label{P1-eq-mr-70}
 |\Gamma_3| =\left|\frac{p_3}{8(1+3\alpha)}- \frac{p_1p_2}{3(1+\alpha)(1+2\alpha)} +\frac{5p_1^3}{24(1+\alpha)^3}\right|.   
\end{equation}
Using \eqref{P1-eq-mr-34} in \eqref{P1-eq-mr-70}, we obtain 
\begin{equation}\label{P1-eq-mr-75}
 |\Gamma_3|
= \dfrac{1}{4(1+3\alpha)}
\left|c_3+ \mu c_1 c_2 + \nu c_1^{3}\right|,   
\end{equation}
where $\mu$ and $\nu$ are defined as 
\begin{align}
    & \mu= 2 - \dfrac{16(1+3\alpha)}{3(1+\alpha)(1+2\alpha)}, \nonumber\\ 
    & \nu= 1 - \dfrac{16(1+3\alpha)}{3(1+\alpha)(1+2\alpha)}
  + \dfrac{20(1+3\alpha)}{3(1+\alpha)^{3}} \label{P1-eq-mr-72}.
\end{align}
We now use Lemma  \ref{p1-lemma-015} to  estimate the third logarithmic coefficient $\gamma_{3}$. For $\alpha \geq 0$, a detailed  and deep  calculation shows that 
 \begin{itemize}
    \item $(\mu, \nu) \in D_1$ when $\alpha \in [1.9854, 4.14184]$.
    \item $(\mu, \nu) \in D_2 $ when $\alpha \in [0.809, 1.9854] \cup [4.1418,\infty)$.
    \item $(\mu, \nu) \in D_5 $ when $\alpha \in [0.7207, 0.809]$.
    \item $(\mu, \nu) \in D_6 $ when $\alpha \in [0, 0.7207]$.
\end{itemize}
This shows that the sets $D_1, D_2, D_5$ and $D_6$ are nonempty and together they cover all possible values of $\alpha \geq 0$.  
Now we consider the following cases:

\noindent \textbf{Case-I}: Let $(\mu, \nu) \in D_1 \cup D_2$, then by Lemma  \ref{p1-lemma-015}
$$\left|\,c_{3}+\mu\,c_{1}c_{2}+\nu\,c_{1}^{3}\right| \leq 1.$$  So from \eqref{P1-eq-mr-75}, we get 
$$
|\Gamma_3| \leq \dfrac{1}{4(1+3\alpha)}~~~ \mbox{for}~~~ \alpha\geq 0.809.
$$
The inequality is sharp for the inverse function of $f_3 \in \mathcal{W}(\alpha)$  defined by \eqref{P1-eq-mr-36}.

\noindent \textbf{Case-II}: Let $(\mu, \nu) \in D_5 \cup D_6$, then by Lemma  \ref{p1-lemma-015}
$$\left|\,c_{3}+\mu\,c_{1}c_{2}+\nu\,c_{1}^{3}\right| \leq |\nu|.$$
So from \eqref{P1-eq-mr-75}, we get 
\begin{equation}\label{P1-eq-mr-80}
|\Gamma_3| \leq \dfrac{1}{4(1+3\alpha)} |\nu|,~~~ \mbox{for}~~~ 0 \leq \alpha < 0.809,
\end{equation}
where $\nu$ is defined by \eqref{P1-eq-mr-72}. By differentiate $\nu$ with respect to $\alpha$, we get 
$$
\nu' = \dfrac{8 \alpha}{3 (1 + \alpha)^4(1 + 2 \alpha)^2} (12 \alpha^3 - 28 \alpha^2 -32 \alpha -7).
$$
It is easy to check that $$
\nu' < 0 ~~~ \mbox{for}~~~ \alpha \in [0, 3.21836]~~~ \mbox{and}~~ \nu' > 0 ~~~ \mbox{for}~~~ \alpha > 3.21836.$$ 
Thus $\nu$ is decreasing for $\alpha \in [0, 3.21836]$. In particular $\nu$ is decreasing for $\alpha \in [0, 0.809]$. Also
$$
\min_{\alpha \in [0, 0.809]} \nu(\alpha) = \nu(0.809)= 0.99978 > 0.  
$$
Hence $\nu >0$  for $\alpha \in [0, 0.809]$. Therefore in view of  \eqref{P1-eq-mr-72}, from \eqref{P1-eq-mr-80}, we get 
\begin{align*}
    |\Gamma_3| & \leq \dfrac{1}{4(1+3\alpha)} \nu\\
    & = \dfrac{1}{4(1+3\alpha)}-\dfrac{4}{3(1+\alpha)(1+2\alpha)}+\dfrac{5}{3(1+\alpha)^3}~~~ \mbox{for} ~~~0 \leq \alpha< 0.809.
    \end{align*}
The inequality is sharp for the inverse of the function $f_1 \in \mathcal{W}(\alpha)$  defined by  \eqref{P1-eq-mr-26}.

\end{proof}


\section{Moduli Difference of Successive Logarithmic and Logarithmic Inverse Coefficients}\label{section-4}

In this section, we aim to determine the sharp bounds of $|\gamma_2| - |\gamma_1|$ (respectively, $|\Gamma_2| - |\Gamma_1|$ ) for $f \in \mathcal{W}(\alpha)$ for $\alpha \geq 0$, as well as for their inverse functions. We now state our first main result of this section.


\begin{theorem}
 Let $f\in \mathcal{W}(\alpha)$ be of the form \eqref{P1-eq-01}. 
 Then
 \begin{align*}
    \bigl|\,\gamma_2 \,\bigr|-\bigl|\,\gamma_1\,\bigr|\le \dfrac{1}{3(1+2\alpha)} ~~~~ \mbox{for} ~~~~ \alpha \geq 0  
 \end{align*}
and 
 \begin{align*}
    \bigl|\,\gamma_2 \,\bigr|-\bigl|\,\gamma_1\,\bigr|\ge\begin{cases}
     -\dfrac{1}{\sqrt{8\alpha^2+10\alpha+5}} & 0 \leq \alpha\le2.232\\
     -\dfrac{1}{3(1+2\alpha)}-\dfrac{3(1+2\alpha)}{4(8\alpha^2+10\alpha+5)} & \alpha>2.232.
 \end{cases} 
 \end{align*}
All these inequalities are sharp.

\end{theorem}
\begin{proof}
 Let $f\in\mathcal{W}(\alpha)$. Then from \eqref{p1-eq-002} we get 
\begin{equation}\label{P1-eq-mr-85}
    \bigl|\,\gamma_2 \,\bigr|-\bigl|\,\gamma_1\,\bigr|= \left|\frac12\Big(a_3-\frac12 a_2^2\Big) \right| - \left| \frac{a_2}{2}\right|.
\end{equation}
In view of \eqref{P1-eq-mr-10}, from \eqref{P1-eq-mr-85}, we get  
\begin{align}
    \bigl|\,\gamma_2 \,\bigr|-\bigl|\,\gamma_1\,\bigr| & = \bigl|\,\frac{p_2}{6(1+2\alpha)}-\frac{p_1^2}{16(1+\alpha)^2}\,\bigr|-\bigl|\,\frac{p_1}{4(1+\alpha)}\,\bigr|\nonumber  \\
    & = |B_2 p_1^2 + B_3 p_2 | - |B_1 p_1| =  \Psi_{+}(p_1,p_2), \label{P1-eq-mr-90}
\end{align}
Where
\begin{equation}\label{P1-eq-mr-91}
   B_1=\frac{1}{4(1+\alpha)},~~B_2=-\frac{1}{16(1+\alpha)^2},~~ \mbox{and} ~~B_3=\frac{1}{6(1+2\alpha)}. 
\end{equation}
For $\alpha \geq 0$, one can quickly verify that
 $$|2B_2+B_3|=\left|\frac{1}{6(1+2\alpha)}-\frac{1}{8(1+\alpha)^2}\right| \ngeq \frac{1}{6(1+2\alpha)}+\frac{1}{4(1+\alpha)} = |B_3|+B_1 $$
Therefore, by Lemma \ref{p1-lemma-005}, we obtain
$$\Psi_{+}(p_1,p_2) \le2B_3=\frac{1}{3(1+2\alpha)}.$$ 
Hence, from \eqref{P1-eq-mr-90}, we get 
$$\bigl|\,\gamma_2 \,\bigr|-\bigl|\,\gamma_1\,\bigr|\le\frac{1}{3(1+2\alpha)}.$$
The equality holds for the function $f_2 \in \mathcal{W}(\alpha)$ defined by \eqref{P1-eq-mr-27}.

Now, to obtain a lower estimate of  $\bigl|\,\gamma_2 \,\bigr|-\bigl|\,\gamma_1\,\bigr|$, we consider
\begin{align}\label{P1-eq-mr-92}
  B_4=|4B_2+2B_3|& =\left|\frac{1}{3(1+2\alpha)}-\frac{1}{4(1+\alpha)^2}\right|. 
  \end{align}
 A detailed calculations show that for $\alpha \geq 0$,   
 $$B_4+2|B_3| = \dfrac{2}{3 (1 + 2\alpha)} - \dfrac{1}{4(1 + \alpha)^2} > \dfrac{1}{4(1 + \alpha)}.$$ 
 Hence $B_1 \geq B_4+2|B_3|$ is not possible for any $\alpha \geq 0$. Further, 
 for $0 \leq \alpha \leq 2.232$, we show that 
$$B_1^2\le2|B_3|(B_4+2|B_3|).$$
So, by  Lemma \ref{p1-lemma-005}, we obtain 
\begin{align*}
 \Psi_{-}(p_1,p_2)\le 
\begin{cases}
\dfrac{2B_1 \sqrt{2|B_3|}}{\sqrt{B_4 + 2|B_3|}}, & \mbox{for}~~0 \leq \alpha\le2.232,\\
2|B_3| + \dfrac{B_1^2}{B_4 + 2|B_3|}, & \mbox{for}~~\alpha>2.232, 
\end{cases}   
\end{align*}
where $B_1, B_2, B_3$ and $B_4$ are defined by \eqref{P1-eq-mr-91} and \eqref{P1-eq-mr-92}. Putting the values of $B_1, B_2, B_3$ and $B_4$, we get 
\begin{align}\label{P1-eq-mr-95}
    & \frac{2B_1 \sqrt{2|B_3|}}{\sqrt{B_4 + 2|B_3|}}=\frac{1}{\sqrt{8\alpha^2+10\alpha+5}}~~~ \mbox{and }\\ \nonumber 
    & 2|B_3| + \frac{B_1^2}{B_4 + 2|B_3|}=\frac{1}{3(1+2\alpha)}+\frac{3(1+2\alpha)}{4(8\alpha^2+10\alpha+5)}.
\end{align}
Also from Lemma \ref{p1-lemma-005}, we know $ \Psi_{-}(p_1,p_2) = -\Psi_{+}(p_1,p_2)$. Therefore, in view of \eqref{P1-eq-mr-95}, from \eqref{P1-eq-mr-90}, we obtain 
$$\bigl|\,\gamma_2 \,\bigr|-\bigl|\,\gamma_1\,\bigr|\ge\begin{cases}
     -\dfrac{1}{\sqrt{8\alpha^2+10\alpha+5}}, ~~~~~~\mbox{for}~~~0 \leq \alpha\le2.232,\\
     -\dfrac{1}{3(1+2\alpha)}-\dfrac{3(1+2\alpha)}{4(8\alpha^2+10\alpha+5)},~~~\mbox{for}~~~\alpha>2.232.
 \end{cases}$$
The inequalities are sharp for the function $f_4 \in \mathcal{W}(\alpha)$ defined by 
$$
f'_4 (z) + \alpha f''_4(z) = p_1(z)
$$
where $p_1 \in \mathcal{P}$ with the following form 
 $$p_1(z)=\begin{cases}
\dfrac{1-z^2}{1 - \dfrac{4(1 + \alpha)}{\sqrt{8 \alpha^2 + 10 \alpha +5}} + z^2},~~~
\mbox{for}~~0 \leq \alpha\le2.232,\\
\dfrac{1-z^2}{1-\dfrac{6(1+2\alpha)(1+\alpha)}{8\alpha^2+10\alpha+5}z+z^2},~~~ \mbox{for}~~~\alpha>2.232.
 \end{cases}$$
\end{proof}

Next, we obtain the sharp upper and lower bounds for $|\Gamma_2| - |\Gamma_1|$ where $f$ belongs to the class $\mathcal{W}(\alpha)$.

\begin{theorem}
 Let $f(z)\in \mathcal{W}(\alpha)$ be of the form \eqref{P1-eq-01} and $\Gamma_1$,$\Gamma_2$ logarithmic inverse coefficients of $f$.
 Then $$\bigl|\,\Gamma_2 \,\bigr|-\bigl|\,\Gamma_1\,\bigr|\le \dfrac{1}{3(1+2\alpha)}, ~~~~ \mbox{for}~~ \alpha \geq 0,$$ 
 and 
 $$\bigl|\,\Gamma_2 \,\bigr|-\bigl|\,\Gamma_1\,\bigr|\ge
 \begin{cases}
-\dfrac{1}{\sqrt{3(1+2\alpha)}}~~~~~~~~~~~~~\mbox{for}~~~~~~~\alpha \in [0, 1.5], \\
-\dfrac{1}{2(1+\alpha)}+\left|\dfrac{3}{4(1+\alpha)^2} -\dfrac{1}{3(1+2\alpha)}\right|~~~~~~~~\mbox{for}~~~~\alpha \in [2, 5.84],\\
~~~M(\alpha)~~~~~~~~~~~~~~\mbox{for}~~~~~~~ \alpha \in (1.5, 2) \cup (5.84 , \infty),    
\end{cases}
$$ 
where $$M(\alpha) = -\dfrac{1}{3(1+2\alpha)}-\dfrac{3(1+2\alpha)}{4|4\alpha^2 - 10 \alpha -5|+16(1+\alpha)^2}.$$
All these inequalities are sharp.
\end{theorem}

\begin{proof}

Let $f\in\mathcal{W}(\alpha)$. Then  from \eqref{eq:Gamma-coefficients}, we get 
\begin{equation}\label{P1-eq-mr-100}
 \bigl|\,\Gamma_2 \,\bigr|-\bigl|\,\Gamma_1\,\bigr|= \left|-\frac12\Big(a_3-\frac32 a_2^2\Big)\right| - \left|-\frac{a_2}{2} \right|.  
\end{equation}
In view of \eqref{P1-eq-mr-10}, from \eqref{P1-eq-mr-100}, we obtain 
\begin{align}\label{P1-eq-mr-105}
   \bigl|\,\Gamma_2 \,\bigr|-\bigl|\,\Gamma_1\,\bigr| & =  \bigl|\,\frac{3p_1^2}{16(1+\alpha)^2}-\frac{p_2}{6(1+2\alpha)}\,\bigr|-\bigl|\,\frac{p_1}{4(1+\alpha)}\,\bigr| =\Psi_{+}(p_1,p_2)
\end{align}
where
\begin{equation}\label{P1-eq-mr-110}
 B_1=\frac{1}{4(1+\alpha)},~~~B_2=\frac{3}{16(1+\alpha)^2}~~\mbox{and}~~~ B_3=-\frac{1}{6(1+2\alpha)}.   
\end{equation}
For $\alpha \geq 0$, it is straightforward to check that  
\begin{align*}
|2B_2+B_3| & = \left|\frac{1}{6(1+2\alpha)}-\frac{3}{8(1+\alpha)^2}\right|\\
& < \frac{1}{6(1+2\alpha)}+\frac{1}{4(1+\alpha)} \\
& =  |B_3|+B_1
\end{align*}
 and hence $|2B_2+B_3|\ngeq|B_3|+B_1.$ So using  Lemma \ref{p1-lemma-005}, we get
$$\Psi_{+}(p_1,p_2) \le2B_3=\frac{1}{3(1+2\alpha)}.$$
Therefore, from \eqref{P1-eq-mr-105}, we get 
$$\bigl|\,\Gamma_2 \,\bigr|-\bigl|\,\Gamma_1\,\bigr|\le\frac{1}{3(1+2\alpha)} ~~~ \mbox{for}~~ \alpha \geq 0$$
and the equality holds for the inverse function of  $f_2 \in \mathcal{W}(\alpha)$ defined by \eqref{P1-eq-mr-27}.

Now, to get a lower estimate of  $\bigl|\,\Gamma_2 \,\bigr|-\bigl|\,\Gamma_1\,\bigr|$, consider $B_4=|4B_2+2B_3|$. By putting the values of $B_2$ and $B_3$ from \eqref{P1-eq-mr-110}, we get 
\begin{align}\label{P1-eq-mr-115} 
  B_4=|4B_2+2B_3|& =\left|\dfrac{3}{4(1+\alpha)^2}-\dfrac{1}{3(1+2\alpha)}\right| \\\nonumber 
  & = \begin{cases}
 \dfrac{3}{4(1+\alpha)^2}-\dfrac{1}{3(1+2\alpha)} ~~~\mbox{for}~~~0\leq \alpha\leq 2.927, \\  
-\dfrac{3}{4(1+\alpha)^2}+\dfrac{1}{3(1+2\alpha)}~~~ \mbox{for}~~~\alpha > 2.927.
\end{cases}
\end{align}
In view of \eqref{P1-eq-mr-110} and \eqref{P1-eq-mr-115}, it is easy to check that  for $0\leq \alpha\leq 2.927$, $$B_1\ge B_4+2|B_3| ~~~~ \mbox{when}~~~\alpha \geq 2$$ and  for $ \alpha >  2.927$ 
$$B_1\ge B_4+2|B_3| ~~~~ \mbox{when}~~~\alpha \leq 5.84.$$
So, as a whole, we get 
\begin{equation}\label{P1-eq-mr-120}
  B_1\ge B_4+2|B_3|  ~~ \mbox{when}~~ 2 \leq \alpha \leq 5.84.
\end{equation} 
Furthermore, by \eqref{P1-eq-mr-115} and \eqref{P1-eq-mr-110}, a calculation of the details shows that for $\alpha > 2.927$
$$
B_1^2 \nleq 2|B_3|(B_4+2|B_3|)
$$
and for $0 \leq \alpha \leq 2.927$ we get 
\begin{equation}\label{P1-eq-mr-125}
   B_1^2\le2|B_3|(B_4+2|B_3|) ~~~ \mbox{when} ~~~ 0\leq \alpha \leq 1.5.
\end{equation}
So using  Lemma \ref{p1-lemma-005}, we obtain 

\begin{align}\label{P1-eq-mr-130}
\Psi_{-}(p_1,p_2)\le 
\begin{cases}
\dfrac{2B_1 \sqrt{2|B_3|}}{\sqrt{B_4 + 2|B_3|}}~~~~~~~ \mbox{when}~~~0 \leq \alpha\leq 1.5,\\
2B_1 - B_4~~~~~~ \mbox{when}~~~2 \leq \alpha\leq 5.84,\\
2|B_3| + \dfrac{B_1^2}{B_4 + 2|B_3|}~~~~\mbox{when}~~~1.5 <\alpha < 2 ~~~\mbox{and} ~~~ \alpha >5.84.
\end{cases}
\end{align}
Also from Lemma \ref{p1-lemma-005}, we know $\Psi_{-}(p_1,p_2) = -\Psi_{+}(p_1,p_2)$. By putting the value of $B_1, B_2, B_3$ and $B_4$ in \eqref{P1-eq-mr-130}, form \eqref{P1-eq-mr-105}, we obtain
$$\bigl|\,\Gamma_2 \,\bigr|-\bigl|\,\Gamma_1\,\bigr|\ge\begin{cases}
-\dfrac{1}{\sqrt{3(1+2\alpha)}}~~~~~~~~~~~~~~~~~~\mbox{for}~~~~~~~0 \leq \alpha\leq 1.5, \\
-\dfrac{1}{2(1+\alpha)}+\left|\dfrac{3}{4(1+\alpha)^2} -\dfrac{1}{3(1+2\alpha)}\right|~~~~~~~~\mbox{for}~~~~~2 \leq\alpha \leq 5.84,\\
~~~M(\alpha)~~~~~~~~~~~~~~\mbox{for}~~~~~~~ \alpha \in (1.5, 2) \cup (5.84 , \infty),    
\end{cases}
$$ 
where $$M(\alpha) = -\dfrac{1}{3(1+2\alpha)}-\dfrac{3(1+2\alpha)}{4|4\alpha^2 - 10 \alpha -5|+16(1+\alpha)^2}.$$
The inequalities are sharp for the inverse  function of $f_5 \in \mathcal{W}(\alpha)$ defined by 
$$
f'_5 (z) + \alpha f''_5(z) = p_2(z)
$$
where $p_2 \in \mathcal{P}$ with the following form 
 $$p_2(z)=\begin{cases}
 \dfrac{1-z^2}{1-\dfrac{4(1+\alpha)}{3\sqrt{1+2\alpha}}z +z^2}~~~~~~~~~~\mbox{for}~~~~~~~0 \leq \alpha\leq 1.5,\\
 \dfrac{1+z}{1-z} ~~~~~~~~~~~~~~~~~\mbox{for}~~~~~2 \leq\alpha \leq 5.84,\\
\dfrac{1-z^2}{1-2\zeta_1z+z^2}~~~~~~~~~\mbox{for}~~~~\alpha \in (1.5, 2) \cup (5.84 , \infty),
 \end{cases}
 $$
where 
$$
\zeta_1=\frac{\dfrac{1}{4(1+\alpha)}}{\frac{1}{3(1+2\alpha)} + \left|\frac{3}{4(1+\alpha)^2}-\frac{1}{3(1+2\alpha)}\right|},~~~~~\alpha \in (1.5, 2) \cup (5.84 , \infty).
$$
\end{proof}


\section{ Hankel Determinant Having Elements of Logarithmic and Logarithmic Inverse Coefficients}\label{section-5}

In this section, we discuss the second Hankel determinant for functions in the class $\mathcal{W}(\alpha)$ where $\alpha = 1$. For $f \in \mathcal{S}$ given by \eqref{P1-eq-01}, the second Hankel determinant of $F_f/2$ using \eqref{p1-eq-002}, is given by 
\begin{equation}\label{P1-eq-mr-135}
    H_{2,1}\!\left(\frac{F_f}{2}\right) :=  \gamma_3 \gamma_1 - \gamma_2^2 = \dfrac14\left(a_2 a_4 - a_3^2 + \dfrac{1}{12}a_2^4\right). 
\end{equation}
Also the second Hankel determinant of $F_{f^{-1}}/2$ using \eqref{eq:Gamma-coefficients}, is given by 
\begin{equation}\label{P1-eq-mr-136}
    H_{2,1}\!\left(\frac{F_{{f}^{-1}}}{2}\right) :=  \Gamma_3 \Gamma_1 - \Gamma_2^2 =\dfrac{1}{4}\left(\dfrac{13}{12}a_2^{4}+ a_2a_4-  a_2^2a_3- a_3^2\right).
\end{equation}

In our first result, we get the sharp bound of $H_{2,1}(F_f/2)$ for functions in $\mathcal{W}(1)$.

\begin{theorem}\label{P1-thm-001}
Let $f\in\mathcal{W}(1)$ be of the form \eqref{P1-eq-01}.  Then 
$$
\left| H_{2,1}\!\left(\frac{F_f}{2}\right) \right| \le \frac{1}{81}.
$$    The bound is sharp.
\end{theorem}

\begin{proof}
 Let $f(z)\in \mathcal{W}(1)$ be of the form  \eqref{P1-eq-01}. Then by the definition of subordination there exists a Schwarz function $\varphi$ defined by \eqref{Schwarz function-1} such that 
 \begin{equation}\label{P1-eq-mr-140}
    f'(z)+z f''(z)=\dfrac{1+\varphi(z)}{1-\varphi(z)}.
 \end{equation}
In view of the representation of power series of $f$ and $\varphi$ and comparing the coefficients on both sides of \eqref{P1-eq-mr-140}, we get 
\begin{equation}\label{P1-eq-mr-145}
a_2 = \frac{c_1}{2}, ~~~~ a_3 = \frac{2 (c_2 + c_1^2)}{9}~~~~ \mbox{and}~~~~~ a_4 = \frac{c_3 + 2c_1c_2 + c_1^3} {8}.
\end{equation}
Substitute the above values of $a_2, a_3$ and $a_4$ in \eqref{P1-eq-mr-135} and  further simplification gives 
\begin{equation}\label{P1-eq-mr-150}
 H_{2,1}\!\left(\frac{F_f}{2}\right) =\dfrac14 \left(\dfrac{95}{5184}c_1^4 + \dfrac{17}{648}c_1^2 c_2 + \dfrac{1}{16}c_1 c_3 - \dfrac{4}{81}c_2^2 \right).  \end{equation}
Taking the modulus and using the triangle inequality, \eqref{P1-eq-mr-150}  yields
\begin{equation}\label{P1-eq-mr-155}
    4\left|H_{2,1}\!\left(\frac{F_f}{2}\right)\right|
 \le \frac{95}{5184}|c_1|^4
+ \frac{17}{648}|c_1|^2|c_2|
+ \frac{1}{16}|c_1||c_3|
+ \frac{4}{81}|c_2|^2.
\end{equation}
By Lemma \ref{p1-lemma-010} from \eqref{P1-eq-mr-155}, we obtain 
\begin{equation}\label{P1-eq-mr-160}
    4\left|H_{2,1}\!\left(\frac{F_f}{2}\right)\right|
 \le \frac{95}{5184}|c_1|^4 + \frac{17}{648}|c_1|^2|c_2| + \frac{1}{16}|c_1|\left( 1 - |c_1|^2 - \frac{|c_2|^2}{1 + |c_1|^2} \right) 
+ \frac{4}{81}|c_2|^2.
\end{equation}
Let $x:=|c_1|$ and $y:=|c_2|$. 
In view of the lemma \ref{p1-lemma-010}, the region of variability of the  pair $(x, y)$ coincides with the set
$$
\Omega: = \{(x, y): 0 \leq x \leq 1, 0 \leq y \leq 1 - x^2\}. 
$$
Therefore form \eqref{P1-eq-mr-160}, we get 
\begin{align}\label{P1-eq-mr-165}
    \left|H_{2,1}\!\left(\frac{F_f}{2}\right)\right| \leq  F(x, y)
\end{align}
where
\begin{equation}\label{P1-eq-mr-170}
 F(x,y ) =  \frac{95}{20736}x^4 + \frac{17}{2592}x^2 y + \frac{1}{64}x\!\left(1-x^2-\frac{y^2}{1+x}\right)
+ \frac{1}{81}y^2. 
\end{equation}
Also the function $F(x, y)$ defined by \eqref{P1-eq-mr-170} can be represented as a quadratic polynomial of $y$ with variable coefficients of $x$ as 
\begin{equation}\label{P1-eq-mr-175}
 F(x, y)=  A(x) + B(x)\,y + C(x)\,y^2,
\end{equation}
where 
$$A(x)=\frac{95}{20736}x^4+\frac{1}{64}x(1-x^2), ~~ B(x)=\frac{17}{2592}x^2 ~~\mbox{and}~~ C(x)=\frac{1}{81}-\frac{x}{64(1 + x)}.$$
We need to find the maximum value of $F(x, y)$ in the region $\Omega$.  
For $0 \leq x \leq 1$,  one can verify that $0\le \frac{x}{1+x}\le\frac12$. Thus 
$$C(x)\ge\dfrac{1}{81}-\dfrac{1}{128}=\dfrac{47}{10368}>0.$$  
So for each fix $x\in [0, 1]$, the quadratic polynomial of $y$ given by \eqref{P1-eq-mr-175} is a convex function on $\Omega$. Therefore, by  Bauer's maximum principle, for each fixed $x\in [0, 1]$, the maximum value  of 
$$A(x)+B(x)y+C(x)y^2 ~~~~~~ \mbox{for} ~~~ 0 \leq y \leq 1 - x^2$$ 
occurs at one of the endpoints, i.e.,  either at   $y=0$ or at $y=1-x^2$. Now we consider the following cases:

\noindent \textbf{Case-I}:
At the end point $y=0$, 
$$
\max_{x \in [0,1]} F(x, y) = \max_{x \in [0,1]} F(x, 0). 
$$
From \eqref{P1-eq-mr-170}, we get 
\begin{align*}
    F(x, 0) = \frac{95}{20736}x^4+\frac{1}{64}x(1-x^2)~~~\mbox{for}~~~0 \leq x\leq 1. 
\end{align*}
A straightforward and detailed calculation shows that 
\begin{align*}
    F(x, 0) \leq  0.01059~~~\mbox{for}~~~0 \leq x\leq 1. 
\end{align*}
So 
\begin{equation}\label{P1-eq-mr-180}
 \max_{x \in [0,1]} F(x, 0)  =  0.01059. 
\end{equation}

\noindent \textbf{Case-II}:
At the end point  $y=1-x^2$,
$$
\max_{x \in [0,1]} F(x, y) = \max_{x \in [0,1]} F(x,  1- x^2). 
$$
From \eqref{P1-eq-mr-170}, we get 
\begin{align*}
    F(x, 1 - x^2) &:=\frac{95}{20736}x^4+\frac{17}{2592}x^2(1-x^2)
+\frac{1}{64}x\left(1-x^2-\frac{(1-x^2)^2}{1+x}\right)
+\frac{1}{81}(1-x^2)^2\\[4pt]
&=-\frac{109}{20736}x^4-\frac{13}{5184}x^2+\frac{1}{81},
\end{align*}
which is a decreasing function for $0 \leq x \leq 1.$ So 
\begin{equation}\label{P1-eq-mr-185}
 \max_{x \in [0,1]} F(x, 1 - x^2)  =  \frac{1}{81}. 
\end{equation}
So, from \eqref{P1-eq-mr-180} and \eqref{P1-eq-mr-185} we get
$$
F(x, y ) \leq \max \{0.0105, \frac{1}{81}\} = \frac{1}{81} ~~~ \mbox{for}~~(x, y) \in \Omega.
$$
Hence in view of \eqref{P1-eq-mr-165}, we obtain 
$$
\left|H_{2,1}\!\left(\frac{F_f}{2}\right)\right| \leq  \frac{1}{81}.
$$
The above inequality is sharp for the function $f_6 \in \mathcal{W}(1)$ defined by
\begin{equation}\label{P1-eq-mr-186}
 f_6(z)=z+\sum_{n=1}^{\infty}\frac{2z^{2n+1}}{(2n+1)^2}.   
\end{equation}
\end{proof}


Next, we will discuss the sharp bound of $H_{2, 1}(F_{f^{-1}}/2)$ for the functions in class $\mathcal{W}(1)$.
\begin{theorem}
\;If $f\in\mathcal{W}(1)$  be of the form \eqref{P1-eq-01}.Then 
$$
\left| H_{2,1}\!\left(\frac{F_{f^{-1}}}{2}\right) \right| \le \frac{1}{81}.
$$    
The inequality is sharp.
\end{theorem}

\begin{proof}
 Let $f(z)\in \mathcal{W}(1)$ be of the form \eqref{P1-eq-01}. Using \eqref{P1-eq-mr-145}, form \eqref{P1-eq-mr-136}, we get  the second Hankel determinant
\begin{align}\label{P1-eq-mr-190}
H_{2,1}\!\left(\frac{F_{f^{-1}}}{2}\right)
& =\Gamma_3\Gamma_1-\Gamma_2^2\\ \nonumber 
& =\dfrac{13}{48}a_2^{4}+\dfrac{1}{4}a_2a_4-\dfrac{1}{4}a_2^2a_3-\dfrac{1}{4}a_3^2\\ \nonumber 
& = \dfrac{131}{20736}c_1^4 - \dfrac{19}{2592}c_1^2 c_2 + \dfrac{1}{64}c_1 c_3 - \dfrac{1}{81}c_2^2. \label{P1-eq-mr-190}
\end{align}
Taking the modulus and using the triangle inequality, \eqref{P1-eq-mr-190} yields
 $$
\left| H_{2,1}\!\left(\frac{F_{f^{-1}}}{2}\right) \right|
\le \frac{131}{20736}x^4
+ \frac{19}{2592}x^2 y
+ \frac{1}{64}x\,\left(1-x^2-\frac{y^2}{1+x}\right)
+ \frac{1}{81}y^2 = F(x, y)
$$
where $x = |c_1|$ and $y = |c_2|$.
Adopting the similar techniques from Theorem \ref{P1-thm-001} and after doing a detailed calculation, we obtain 


$$
\left| H_{2,1}\!\left(\frac{F_{f^{-1}}}{2}\right) \right| \le \frac{1}{81}.
$$
The equality holds for inverse function of $f_6 \in  \mathcal{W}(1)$ defined by \eqref{P1-eq-mr-186}.

\end{proof}

\section{Declarations}
\textit{Acknowledgments:} The first author would like to thank UGC, Govt. of India, for the financial support (NTA Ref. No. 231620097064 ) in the form of a fellowship.\\
\textit{Author Contributions:} All the authors contributed equally to this manuscript and reviewed it. \\
 \textit{Data Availability :} No datasets were generated or analysed during the current study. \\
\textit{Conflict of interest:} There is no competing interest.\\

\bibliographystyle{amsplain}
\bibliography{references}

\end{document}